\documentclass{article}
\usepackage{amsmath,xypic,amssymb,amsthm}

\DeclareMathOperator{\Hom}{Hom}
\DeclareMathOperator{\HH}{HH}
\DeclareMathOperator{\End}{End}
\DeclareMathOperator{\ext}{ext}
\DeclareMathOperator{\Ext}{Ext}
\DeclareMathOperator{\Def}{Def}
\DeclareMathOperator{\Spec}{Spec}
\DeclareMathOperator{\Simp}{Simp}
\DeclareMathOperator{\Id}{Id}
\DeclareMathOperator{\Inner}{Inner}
\DeclareMathOperator{\Der}{Der}
\DeclareMathOperator{\Ker}{Ker}

\DeclareMathOperator{\Sets}{Sets}
\DeclareMathOperator{\Ima}{Im}

\DeclareMathOperator{\ob}{ob}

\DeclareMathOperator{\rad}{rad}
\DeclareMathOperator{\uc}{\underline{c}}

\DeclareMathOperator{\ut}{\underline{t}}
\DeclareMathOperator{\uu}{\underline{u}}

\DeclareMathOperator{\ad}{ad}

\define\s#1{\mathcal #1}

\newtheorem{proposition}{Proposition}

\newtheorem{lemma}{Lemma}

\newtheorem{definition}{Definition}

\newtheorem{example}{Example}

\title{Deformation of Diagrams}

\begin{document}
\maketitle

\begin{abstract}

In this this paper we introduce entanglement among the points in a non-commutative scheme, in addition to the tangent directions. A diagram of $A$-modules is a pair $\uc=(|\uc|,\Gamma)$ where $|\uc|=\{V_1,\dots,V_r\}$ is a set of $A$-modules, and $\Gamma=\{\gamma_{ij}(l)\}$ is a set of $A$-module homomorphisms $\gamma_{ij}(l):V_i\rightarrow V_j$, seen as the $0$'th order tangent directions. This concludes the discussion on non-commutative schemes by defining the deformation theory for diagrams, making these the fundamental points of the non-commutative algebraic geometry, which means that the construction of non-commutative schemes is a closure operation. Two simple examples of the theory are given: The space of a line and a point, which is a non-commutative but untangled example, and the space of a line and a point on the line, in which the condition of the point on the line gives an entanglement between the point and the line.

\end{abstract}

\section{Introduction}

In the article \cite{Laudal03} Laudal defines the non-commutative deformation functor $\Def_V:a_r\rightarrow\Sets$, see also Eriksen \cite{Eriksen03}. Here $a_r$ is the category of $r$-pointed, Artinian $k$-algebras $S$ fitting into the diagram $$\xymatrix{k^r\ar[r]\ar[dr]_{\Id}&S\ar[d]^{\rho}\\&k^r.}$$ In $\cite{Siqveland111}$, we define a non-commutative scheme theory, generalizing the commutative one in the geometric situation: Let $A$ be a $k$-algebra, $k$ algebraically closed of characteristic $0$, not necessarily commutative. Let $M$ be a simple (right) $A$-module, and let $\frak m_M$ be the corresponding (right) ideal. $A$ is called geometric if $0=\rad(A)^\infty=\bigcap_{\underset{n\geq 1}{M\in\Simp(A)}}\frak m_M^n$.
In $\cite{Laudal03}$, Laudal proves that a pro-representing hull for the non-commutative deformation functor of $V=\{V_1,\dots,V_r\}$ exists when $V$ is a family of finite dimensional (right or left) $A$-modules. This is a $k$-algebra $\hat{H}=(\hat{H}_{ij})_{1\leq i,j\leq r}$ in the pro-category $\hat{a}_r$ together with a pro-versal (also called mini-versal) family $$ A\overset\iota\rightarrow(\hat{H}_{ij}\otimes_k\Hom_k(V_i,V_j))=\hat{\mathcal O}_V.$$ First of all, the property of $A$ being geometric assures that the pro-versal morphism $\iota$ is injective. Secondly, $\hat{\mathcal O}_V\twoheadrightarrow\oplus_{i=1}^r\End_k(V_i,V_i)$, and it is known that this surjection implies that, as sets, $\Simp(\hat{\mathcal O}_V)=V$. Thus the sub $k$-algebra $\mathcal O_V\subseteq\hat{\mathcal O}_V$ generated by the image of the generators of $A$ and the inverses of the generated elements not in any corresponding maximal ideal is the localization of $A$ in $V$: It is a fractional $k$-algebra of a finitely generated $k$-algebra, and the only simple modules are the modules in $V$ (or equivalently, the only maximal ideals are the maximal ideals corresponding to the modules in $V$).

On the set $\Simp(A)$ we now pose the following saturated Zariski topology: First of all, the Zariski topology is the topology generated by the open base, over $f\in A$, $D(f)=\{V\in\Simp(A)|\rho(f)\in\End_k(V)\text{ is injective}\},$ where $\rho$ is the structure morphism. We let the saturation relation be the equivalence relation generated by the condition that $V_i$ and $V_j$ are related if $\Ext^1_A(V_i,V_j)\neq 0$. This means that an open subset is saturated with all related points, and it is straight forward to prove that this gives a topology. It should also be mentioned that the saturated topology is introduced to ease the notation.

Just as in commutative situation, we define a sheaf of rings, the structure sheaf, on $\Simp(A)$ by

$$\mathcal O(U)=\underset{\underset{\uc\subseteq U}\longleftarrow}\lim \mathcal O_{\uc},$$ where the limit is taken over equivalence classes $\uc$ with respect to the equivalence relation above. Writing out this definition, we see that it is a true generalization of the definition given in Hartshorne \cite{Hartshorne77} for commutative schemes.

Inspired by quantum mechanics, we need to study entangled systems. The points in moduli (the simple modules) is most frequently entangled by not just the tangent directions and higher order momenta, but also directly. This means that the equivalence relation above should include a zero'th derivative, that is elements in $\Hom_A(V_i,V_j)$. So, we define a diagram as a pair $\uc=(|\uc|,\Gamma)$ where $|\uc|=\{V_1,\dots,V_r\}$ is a set of $A$-modules, and $\Gamma=\{\gamma_{ij}(l)\}$ is a set of $A$-module homomorphisms $\gamma_{ij}:V_i\rightarrow V_j,$ for example (where the arrows not necessarily commutes):
$$\xymatrix{c_1\ar[rr]^{\gamma_{12}}\ar[dr]_{\gamma_{13}}&&c_2\ar[dl]^{\gamma_{23}}\\&c_3.&}$$

Extending the equivalence relation demands a generalization of the category $a_r$ and its deformation functor. Together with a couple of more or less trivial examples, this is the main result of the text.

\section{Algebras over $k^r$ and their geometry}

In the commutative algebraic geometry the standard $k$-algebra is the free polynomial algebra $S=k[x_1,\dots,x_d]$ over an algebraically closed field $k$. The reason for this, vied algebraically, is that it is the local model for the charts $\mathbb C^d$ in ordinary differential geometry. Its geometry is obvious: The points are the maximal ideals $(x_1-a_1,\dots,x_d-a_d)$ corresponding to the points $(a_1,\dots,a_d)$. The local charts in differential geometry, modeling our real (or in fact complex) world, is then substituted by the affine algebraic varieties. Those are the set of maximal ideals in the quotient algebras $S/\frak a$ where $\frak a$ is the ideal generated by $r$ polynomials $f_1,\dots,f_r$ in $S$ (as $S$ is noetherian), i.e. $\frak a=(f_1,...,f_r).$ These maximal ideals $\frak m$ necessarily contains $\frak a$, so that $\frak m=(x_1-a_1,\dots,x_d-a_d)$, where $f_1(a_1,\dots,a_d)=\dots=f_r(a_1,\dots,a_d)=0.$ In differential geometry, the topology is the topology that makes holomorphic functions continuous, that is the ordinary Euclidean topology. In algebraic geometry, as we said, our functions are the polynomials in $S$, and so our topology should be the one that makes the polynomials continuous. It is well known that the Zariski topology is the coarsest possible for that matter.

What is the tangent space in a point of $\Spec(S)$? Again, in differential geometry it is the linear space of directions where one can measure linear growth, thus a derivation is a sum of the partial derivatives $\partial=\alpha_1\partial_1+\cdots+\alpha_d\partial_d$, summed over the directions spanning the tangent space. This translates to algebraic geometry by defining the tangent space in perfect analogy: The tangent space is the vector space of derivations $d:S\rightarrow k$,  that is $d=\alpha_1 d_{x_1}+\dots+\alpha_r d_{x_r}$. Before generalizing, note that the above partition of derivations follows from the fact that every derivation is determined by its value one the base of the tangent space, so that the above factoring just says $d(x_i)=\alpha_i$, $i=1,\dots,r$.

To make the generalization to the non-commutative situation, we choose to exploit the correspondence between maximal ideals and simple modules. This is convenient when it comes to the definition of localization: In differential geometry, the local ring of functions at a point is exactly that; the functions defined immediately close to the point. Translating to algebraic geometry, this is the same, and in the Zariski topology the local function ring will be isomorphic to the ring $S$ localized in the maximal ideal $\frak m$, i.e. $\mathcal O_{\Spec S,\frak m}\cong S_{\frak m}$. It is tempting to try to generalize the localization process to non-commutative rings. This has limited success if we use the algebraic definition of localization, but thinking of small perturbations of the functions near the point of interest works very well, the algebraic counterpart of considering continuous perturbations is the process of flat deformations. Thus the localization process is substituted by the deformation theory, the local ring in a point corresponding to a simple module is replaced by the local formal moduli in that point. True enough, this is a formal $k$-algebra, but contains all local information. In \cite{Siqveland111} we proved that this gives a true generalization of the commutative case.

We have defined the category $a_r$ of $r$-pointed Artinian $k$-algebras. The objects in this category fits in the diagram $$\xymatrix{k^r\ar[r]^\iota\ar[dr]_\Id&S\ar[d]^\rho\\&k^r,}$$ with $\rad(S)^n=(\ker\rho)^n=0.$ We use the notation $e_i=\iota(e_i)$ for short, and notice that any element $s\in S$ can be written $$s=1\cdot s\cdot 1=(\sum_{i=1}^re_i)s(\sum_{j=1}^re_j)=\sum_{i=1}^r\sum_{i=1}^re_ise_j=\sum_{i=1}^r\sum_{i=1}^rs_{ij},$$ where we have put $s_{ij}=e_ise_j$. Letting $S_{ij}=e_iSe_j$, we have a $k^r$-algebra homomorphism $\phi:S\rightarrow (S_{ij})_{1\leq i,j\leq r}$ given by $\phi(s)=\sum_{1\leq i,j\leq r}s_{ij}$ which obviously is an isomorphism. So for algebras $S$ in $a_r$, we consider only matrix algebras of the type $(S_{ij})$. Then it is  clear what could be the archetypical $k^r$-algebra:

For each pair $(i,j)$, $1\leq i,j\leq r$, let $t_{ij}(l),$ $1\leq l\leq d_{ij}$ be free matrix variables. Then each $t_{ij}(l)$ is a $r\times r$-matrix, spanning together with the idempotents $e_1,\dots, e_r$, a vector space $W$, and we let the free non-commutative polynomial algebra be the tensor algebra of $W$ over $k^r$. For example, for $r=2$, $d_{ij}=1$, $S=\left(\begin{matrix}k[t_{11}]&t_{12}\\t_{21}&k[t_{22}]\end{matrix}\right),$ and $S$ is the $k^2$-algebra generated by the elements $$e_1=\left(\begin{matrix}1&0\\0&0\end{matrix}\right), e_2=\left(\begin{matrix}0&0\\0&1\end{matrix}\right),$$ $$t_{11}=\left(\begin{matrix}t_{11}&0\\0&0\end{matrix}\right),
t_{12}=\left(\begin{matrix}0&t_{12}\\0&0\end{matrix}\right),t_{21}=\left(\begin{matrix}0&0\\
t_{21}&0\end{matrix}\right),t_{22}=\left(\begin{matrix}0&0\\
0&t_{22}\end{matrix}\right).$$

Recall the following, proved in e.g. \cite{Siqveland111}:
\begin{lemma}\label{Fausklemma} let $R$ be a $k$-algebra, $k$ algebraically closed, and let $V$ be a finite dimensional $R$-module. Then $V$ is simple if and only if the structure morphism $$\rho:R\rightarrow\End_k(V),$$ sending $r\in R$ to $\rho(r)(v)=r\cdot v,$ is surjective.
\end{lemma}

Let $S_{ii}=k\langle t_{ii}(1),\dots,t_{ii}(d_{ii})\rangle$. Then there is a surjection $\rho_{ii}:S\rightarrow S_{ii}$ sending $e_i$ to $1$, $t_{ii}(l)$ to $t_{ii}(l)$ and all other generators to $0$.

\begin{lemma} Let $V$ be a simple (right) $S$-module. Then $S\simeq S_{ii}/\frak m_{ii}=V_{ii}$ where $\frak m_{ii}$ is maximal in $S_{ii}$. So the simple $S$-modules are the $S$-modules on the diagonal.
\end{lemma}

\begin{proof} For a maximal ideal $\frak m_{ii}\subset S_{ii}$, we have an isomorphism $$S/\rho_{ii}^{-1}(\frak m_{ii})\overset\simeq\rightarrow S_{ii}/\frak m_{ii}.$$ This proves that $V_i$ is a simple $S$-module. For the converse, assume $\frak m\subset S$ is maximal. If $\rho_{ii}(\frak m)=S_{ii}$ for all $i$, it follows that $1=\sum e_i$ is in $\frak m$ which is impossible. Thus there exists an $i$ where $\rho_{ii}(\frak m)\subseteq\frak m_{ii}$ for a maximal ideal $\frak m_{ii}\subset S_{ii}$. Then $\frak m\subseteq\rho_{ii}^{-1}(\rho_{ii}(\frak m))\subseteq
\rho_{ii}^{-1}(\frak m_{ii})\subseteq S.$ Then $\frak m=\rho_{ii}^{-1}(\frak m_{ii}),$ and the lemma is proved.
\end{proof}

Now, let $\Simp(S)$ be the set of simple modules that we have found above. We generalize the Zariski topology to the non-commutative case by the following: For an element $s\in S$ we define the subset $D(s)=\{V\in\Simp(S)|\rho(s):V\rightarrow V\text{ is invertible}\}$. The non-commutative Zariski topology is the topology generated by the sets $D(s)$, $s\in S$.

Having the non-commutative space, we also generalize the tangent space in a point, that is in a simple module $V$. As in the commutative case, the tangent space in $V$ is the $k$-vector space of derivations $$T_{\Simp(S),V}=\Der_k(S,\End_k(V,V))/\Inner,$$ where $\Inner$ is the set of inner derivations. This is to say, derivations in a direction is independent on the magnitude of the base vector. We see that in the case of the free $k^r$-algebra $S$, every simple module is one-dimensional, so that $T_{\Simp(S),V}=\Der_k(S,k)/\Inner.$ It is important to notice that $k$ is an $S$ bimodule in a special way: $k\cong\Hom_k(k,k)$, and for any two $S$-modules $V_1$, $V_2$, the left and right actions of $s\in S$ on the bimodule $\Hom_k(V_1,V_2)$ is given respectively by the left and right skew morphism in the diagram

$$\xymatrix{V_1\ar[rd]\ar[r]^\phi&V_2\ar[d]^{\cdot s}\\V_1\ar[u]^{\cdot s}\ar[ur]&V_2,}$$ i.e. $(s\cdot\phi)(v)=\phi(vs)$ and $(\phi\cdot s)(v)=\phi(v)s$.

For a bimodule $M$ the subspace $\Inner\subseteq\Der_k(S)$ is defined as the image of $$M\rightarrow\Der_k(S,M),\text{}m\mapsto\ad(m),$$ where $\ad(m)(s)=ms-sm.$ In this particular case, the inner derivations are easily computed in a point $V_{nn}(p)$: Any derivation is determined on its action on the radical, that is its action on each $t_{ij}(l)$. For an element $\alpha\in k$ we find $$\ad(\alpha)(t_{ij}(l))=\alpha\cdot t_{ij}(l)-t_{ij}(l)\cdot\alpha=0$$ We should also notice that for any derivation (in particular for the inner derivations), for any idempotent we have $\delta(e_i)=\delta(e_i^2)=2\delta(e_i)\Rightarrow\delta(e_i)=0$. So for the free $k^r$-algebra $S$, $T_{\Simp(S),V}=\Der_k(S,k)$ and our generalized result is then that $T_{\Simp(S), V_{ii}(p)}=\oplus_{n=1}^{d_{ii}}k\cdot d_{t_{ii}(n)-p_n}$, where $p=(p_1,\dots,p_{d_{ii}})$ is a point in $S_{ii}=k[t_{11}(1),\dots,t_{11}(d_{ii})]$.

In the commutative situation, for a commutative $k$-algebra $A$ with two different simple $A$-modules $V_1=A/\frak m_1$, $V_2=A/\frak m_2$ it is well known that  $\Ext^1_A(V_1,V_2)\cong\Der_k(V_1,V_2)/\Inner=0$.
In the non-commutative case however, this is is different. The non-commutative information is contained in the different  tangent spaces and higher order derivations between the different points. For simplicity, we give the following definition in all generality, even if it makes sense only for non-commutative $k$-algebras.

\begin{definition}
Let S be any $k$-algebra. The tangent space between two $S$-modules $M_1$ and $M_2$ is $$\Ext^1_A(M_1,M_2)\cong\HH^1(A,\Hom_k(M_1,M_2))$$ where $\HH^\cdot$ is the Hochschild cohomology.
\end{definition}

\begin{example}
Let $S=\left(\begin{matrix}k[t_{11}]&k\cdot t_{12}\\k\cdot t_{21}&k[t_{22}]\end{matrix}\right)$ and consider two general points $V_1=k[t_{11}]/(t_{11}-a)$, $V_2=k[t_{22}]/(t_{22}-b)$. First, we compute $\Ext^1_A(V_i,V_j)\cong\Der_k(V_i,V_j)/\Inner(V_i,V_j)$ by derivations.
\vskip0,1cm
\noindent $\Ext^1_S(V_1,V_1)$: Let $\delta\in\Der_k(S,\End_k(V_1))$. Then $$\delta(e_i)=\delta(e_i^2)=2\delta(e_i)\Rightarrow\delta(e_i)=0, \text{}i=1,2.$$
$$\delta(t_{12})=\delta(t_{12}e_2)=\delta(t_{12})e_2=0,$$
$$\delta(t_{21})=\delta(e_2t_{21})=e_2\delta(t_{21})=0,$$
$$\delta(t_{22})=\delta(t_{22})e_2=0,$$ and finally $$\delta(t_{11})=\alpha.$$ As all inner derivations are zero (easily seen from the computation above), we find that $\Ext^1_S(V_1,V_1)$ is generated by the derivation sending $t_{11}$ to $\alpha$, and all other generators to $0$.
\vskip0,1cm
\noindent $\Ext^1_S(V_1,V_2)$:

For $\delta\in\Der_k(S,\End_k(V_1,V_2))$ things are slightly different.
$\delta(e_1)=\delta(e_1^2)=e_1\delta(e_1)+\delta(e_1)e_1=\delta(e_1)$, that is, the above trick don't work quite the same way. However, for every derivation $\delta:S\rightarrow\End_k(V_1,V_2),$ we find $\delta(e_1)=\alpha$, $\delta(e_2)=-\alpha$,

$$\begin{aligned}
\delta(e_1)&=\alpha,\text{ }\delta(e_2)=-\alpha,\\
\delta(t_{11})&=\delta(t_{11}e_1)=\delta(t_{11})e_1+t_{11}\delta(e_1)=a\alpha,\\
\delta(t_{21})&=\delta(t_{21}e_1)=\delta(t_{21})e_1=0,\\
\delta(t_{22})&=\delta(e_2t_{22})=\delta(e_2)t_{22}=-b\alpha\\
\delta(t_{12})&=\rho\end{aligned}$$

So a general derivation can be written, the $^\ast$ denoting the dual, $$\delta=\alpha e_1^\ast-\alpha e_2^\ast+a\alpha t_{11}^\ast-b\alpha t_{22}^\ast+\rho t_{12}^\ast.$$ For the inner derivations, we compute
$$\begin{aligned}
\ad_\beta(e_1)&=\beta e_1-e_1\beta=-\beta\\
\ad_\beta(e_2)&=\beta e_2-e_2\beta=\beta\\
\ad_\beta(t_{11})&=-\beta a\\
\ad_\beta(t_{22})&=\beta b,
\end{aligned}$$

saying that $$\ad_\beta=\gamma e_1^\ast-\gamma e_2^\ast+a\gamma t_{11}^\ast-b\gamma t_{22}^\ast,\text{ where we have put }\gamma=-\beta.$$
So as $\ad_\beta(t_{12})=0$, and there are no conditions on $\delta(t_{12})$, we get

$$\Ext^1_S(V_1,V_2)=k\cdot t_{12}^\ast=k\cdot d_{t_{12}}.$$ The cases $\Ext^1_S(V_2,V_1)$ and $\Ext^1_S(V_2,V_2)$ are exactly similar.
\end{example}

Generalizing the computation in the above example, we have proved the following:

\begin{lemma} Let $S$ be the general, free, $k^r$-algebra, and let $V_i=V_{ii}(p_{ii})$ be the point $p_{11}$ in entry $i,i$.
Then the tangent space from $V_i$ to $V_j$ is $\Ext^1_S(V_i,V_j)=\oplus_{l=1}^{d_{ij}} k\cdot d_{t_{ij}(l)}.$
\end{lemma}

Now, we will explain what happens in the case with relations, that is, quotients of the free $k^r$-algebra.

\begin{example} We let $R=\left(\begin{matrix}k[t_{11}]&k\cdot t_{12}\\k\cdot t_{21}&k[t_{22}]\end{matrix}\right)/(t_{11}t_{12}-t_{12}t_{22}).$ The polynomial in the ideal is really in the entry $(1,2)$, but there is no ambiguity writing it like this. The points are still the simple modules along the diagonal, but a derivation $\delta\in\Der_k(R,\Hom_k(V_{ii}(p_{ii}),V_{jj}(p_{jj})))$, must this time respect the quotient;$$\delta(t_{11}t_{12}-t_{12}t_{22})=0.$$ This says $$\delta(t_{11}t_{12}-t_{12}t_{22})=t_{11}\delta(t_{12})+\delta(t_{11})t_{12}-t_{12}\delta(t_{22})-\delta(t_{12})t_{22}=0,$$ and is fulfilled for any $\delta\in\Ext^1_R(V_i,V_j)$, $(i,j)\neq (1,2)$. When $\delta\in\Ext^1_R(V_1,V_2)$, we get that the above equation is equivalent to $$t_{11}\delta(t_{12})-\delta(t_{12})t_{22}=\delta(t_{12})(t_{11}-t_{22})=0.$$ Thus in the case that $p_{11}\neq p_{22}$ the tangent direction is annihilated: This quotient has no tangent direction from $V_1(p_1)$ to $V_2(p_2)$ unless $p_1=p_2$.
\end{example}

This example illustrates the geometry of $k^r$-algebras, and is of course nothing else than the obvious generalization of the the ordinary tangent space:

\begin{lemma}
Let $S$ be a finitely generated $k^r$-algebra with residue $\rho:S\rightarrow k^r$ and radical $\frak m=\ker\rho.$ Let $p_1$, $p_2$ be two points on the diagonal of $S$ with respective quotients $V_1\cong V_2\cong k$. Then $T_{p_1,p_2}=\Ext^1_S(V_1,V_2)=\Hom_k(\frak m/\frak m^2,k)$ where the action on $k\cong\Hom_k(V_1,V_2)$ is the left-right action defined by $(s\cdot\phi)(v)=\phi(v\cdot s)$, $(\phi\cdot s) (v)=\phi(v)\cdot s.$
\end{lemma}

The tangent space is not enough to reconstruct the algebra, not even in the commutative situation. As always, to get the full geometric picture we also need the higher order derivatives, the higher order momenta. Even if we cannot reconstruct the algebra in all cases, we get a algebra that is geometrically equivalent (Morita equivalent), and that suffices in construction of moduli.

\section{Higher order derivatives: Generalized Matric Massey Products}

in \cite{Eriksen03} Eriksen has given the description of non-commutative deformation functor. In \cite{Siqveland111} we have defined the generalized matric Massey products. We recall the parts necessary to make the generalization:

Let $A$ be a $k$-algebra. A deformation $M_S$ of an $A$-module $M$ to an Artinian local $k$-algebra $S$ with residue field $k$, i.e. $S\in\ob(\ell)$, is an $S\otimes_k A$-module, flat over $S$, such that $k\otimes_S M_S\cong M$. Two deformations $M_S$ and $M_S^\prime$ are equivalent if there exists an isomorphism $\phi:M_S\rightarrow M_S^\prime$ commuting in the diagram $$\xymatrix{M_S\ar[rr]^\phi\ar[dr]&&M_S^\prime\ar[dl]\\&M&.}$$  This gives the deformation functor $\Def_M:\ell\rightarrow\Sets$ satisfying Schlessinger's wellknown criteria \cite{Schlessinger68}.

The flatness of $M_S\in\Def_M(S)$ over $S$ is equivalent with the fact that as $S$-module, $M_S\cong S\otimes_k M$. For a small surjective morphism $0\rightarrow I \rightarrow S\overset\pi\twoheadrightarrow R\rightarrow 0$, we use induction and linear algebra on the exact sequence $0\rightarrow I\rightarrow S\rightarrow k\rightarrow 0$ to see this. So to give an $A$-module structure on $M_S$ that is a lifting of the $R$-module structure, is equivalent to give a $k$-algebra homomorphism $\sigma_S: A\rightarrow\End_k(M_S)$ commuting in the diagram $$\xymatrix{A\ar[r]^{\sigma_S}\ar[dr]_\sigma&\End_k(M_S)\ar[d]\\&\End_k(M_R).}$$

Using the fact that $\sigma_S$ should commute with the action of $S$, that is, it should be $S$-linear, it is sufficient to define $\sigma_S(a):M\rightarrow S\otimes_k M$. For each $a\in A$, $\sigma_S$ should be a lifting of $\sigma_R$, and so we choose the obvious lifting of $\sigma_R$. Then all properties but the associativity are fulfilled, and the associativity of $\sigma_S$ says $\sigma_S(ab)-\sigma_S(a)\sigma_S(b)=0$. So our obstruction for lifting $M_R$ are the elements $\sigma_S(ab)-\sigma_S(a)\sigma_S(b)\in\Hom_k(M,M\otimes_k I)\cong\End_k(M,M)\otimes_k I$. As these are Hochschild two-cocycles, we have our obstruction $$o(M_R,\pi)\in\HH^2(A,\End_k(M,M))\otimes_kI,$$ with the property that $M_R$ can be lifted to a $M_S$ if and only if $o(M_R,\pi)=0$.

Then we have the an alternative way of viewing this: Choose a free resolution of the $A$-module $M$, $$0\leftarrow M\overset{\varepsilon}\leftarrow L_0\overset{d_0}\leftarrow L_1\overset{d_1}\leftarrow L_2\overset{d_2}\leftarrow\cdots.$$ we have proved that to give a lifting of $M$ to $S$ is equivalent to give a lifting of complexes

$$\xymatrix{&0\ar[d]&0\ar[d]&0\ar[d]&0\ar[d]&\\
0&I\otimes_k M\ar[l]\ar[d]&I\otimes_k L_0\ar[l]\ar[d]&I\otimes_k L_1\ar[l]\ar[d]&I\otimes_k L_2\ar[l]\ar[d]&\cdots\ar[l]\\
0&M_S\ar[l]\ar[d]&S\otimes_k L_0\ar[l]\ar[d]&S\otimes_k L_1\ar[l]\ar[d]&S\otimes_k L_2\ar[l]\ar[d]&\cdots\ar[l]\\
0&M\ar[l]\ar[d]&L_0\ar[l]\ar[d]&L_1\ar[l]\ar[d]&L_2\ar[l]\ar[d]&\cdots\ar[l]\\
&0&0&0&0}$$

For $k[\varepsilon]=k[x]/(x^2)$ as usual, i.e. $\varepsilon^2=0$, the tangent space of the deformation functor is $$\Def_M(k[\varepsilon])\cong\Ext^1_A(M,M)\cong\HH^1(A,\End_k(M)),$$ and likewise for the obstruction space. Using this, we find the correspondence $\Ext^1_A(M,M)\overset\phi\rightarrow\HH^1(A,\End_k(M))$ given as follows: Given a representative $\xi\in\Hom_A(L_1,M)$ for $\overline{\xi}\in\Ext^1_A(M,M)$. Choose a $k$-linear section $\sigma:V\rightarrow L_0$ and let $x\in L_1$ map to $\sigma(am)-a\sigma(m)\in L_0$. Then $\phi(\overline{\xi})(a)(m)=\xi(x)$. It is reasonable to believe that $\phi$ is an isomorphism, and indeed it is. This also generalizes to any higher order derived functors: Both $Ext^1_A(M,-)$ and $\HH^1(A,\Hom_k(M,-)$ are universal delta functors, agreeing on the first term, thus isomorphic.

So we can equally well work in the Yoneda complex, with homomorphisms $d_i:L_i\rightarrow L_{i-1}$ being matrices, as each $L_i$ is assumed to be free. In the above diagram, choose the obvious liftings $d_i^S$ to $S$. Then the obstruction for lifting $M_R$ to $M_S$ via $\pi$ is represented by $$o(M_R,\pi)=\{d_{i-1}^Sd_i^S\}\in\Hom^2(L.,L.)\otimes_k I$$ which is a $2$-cocycle in the Yoneda complex. This is the theory we are going to generalize to the $r$-pointed situation.

Consider then the category of $r$-pointed, Artinian $k$-algebras which we treated above. Let $\{V_1,\dots,V_r\}$ be a set of $r$ $A$-modules and put $V=\oplus_{i=1}^r V_i$. Then a deformation $V_S$ of $V$ to $S$ is a $S\otimes_k A$-module $V_S$, flat over $S$, such that $k^r\otimes_S V_S\cong V$ as $A^r$-modules. As before, two deformations $V_S$ and $V_S^\prime$ are equivalent, if there exists an isomorphism $\phi$ of $S\otimes_k A$-modules commuting in the diagram $$\xymatrix{V_S\ar[rd]\ar[rr]^\phi&&V_S^\prime\ar[dl]\\&V.&}$$

As in the commutative situation, we let flatness (and we can prove that it in fact is) be equivalent to,  as left $S$-module, $$V_S\cong S\otimes_{k^r}V\cong (S_{ij}\otimes_k V_j)_{1\leq i,j\leq r}.$$

Given a small, surjective homomorphism $0\rightarrow I\rightarrow S\overset\pi\rightarrow R\rightarrow 0$ in $a_r$. To give a lifting of $V_R$ to $V_S$ is to give an $A^r$-module structure on $(S_{ij}\otimes_k V_j)_{1\leq i,j\leq r}$, lifting the action on $V_R$, which is a $k$-algebra homomorphism $\sigma_S$ commuting in the diagram $$\xymatrix{A^r\ar[r]^-{\sigma_S}\ar[dr]_{\sigma_R} &\End_k((S_{ij}\otimes_k V_j))\ar[d]\\&\End_k((R_{ij}\otimes_k V_j)).}$$
As the $A$-action is assumed to commute with the $S$ and $R$-actions, by associativity, this is to give, for each $a\in A^r$, a $k$-linear homomorphism $\sigma_a:V\rightarrow (S_{ij}\otimes_k V_j)$. Also, as for each idempotent $e_i\in S$, $\sigma_a(e_i\cdot v)=e_i\sigma_a(v)$, this is equivalent to give a $k$-linear homomorphism $\sigma_a:V_i\rightarrow S_{ij}\otimes_k V_j$ for each $a\in A$. Using this exactly as in the commutative situation, we get the natural $k$-linear lifting of $\sigma_R$ to $S$, everything is fulfilled but the associativity, and we get an obstruction $$o(V_R,\pi)=(o_{ij})\in(\HH^2(A, \Hom_k(V_i,V_j)\otimes_k I_{ij})),$$ where $I=(I_{ij})$ is the kernel of $\pi$, such that $V_R$ can be lifted to $V_S$ if and only if $o(V_R,\pi)=0$.

We have to replace $k[\varepsilon]$ in the $r$-pointed situation. The new basic element in $a_r$ is denoted \textit{the test algebra}, and is not surprisingly given as
$$k[\varepsilon]=\left(\begin{matrix}k\langle t_{11}\rangle&\cdots&kt_{1r}\\\vdots\ddots\vdots\\kt_{r1}&\cdots&k\langle t_{rr}\rangle\end{matrix}\right)/(t_{ij})^2.$$

The tangent space of the deformation functor is then $\Def_{V}(k[\varepsilon_{ij}])$, and again it can be seen that this is isomorphic to the matrix $(\HH^1(A,\Hom_k(V_i,V_j)))$.

To find the correspondence as above, we use free resolutions: For each $V_i$ we choose free resolutions $0\leftarrow V_i\leftarrow L.^i$ with differential $d.^i$, we put $L.=\oplus_{i=1}^r L.^i$, and think of this as a free resolution of $V$ with differential $$d.=\left(\begin{matrix}d.^1&\cdots&0\\\vdots&\ddots&\vdots\\0&\cdots&d.^r\end{matrix}\right).$$ Any morphism $\phi:L_i\rightarrow L_{i-1}$ can be represented by a matrix $\phi=(\phi_{ij})^T$ where $\phi_{ij}:V_i\rightarrow V_j$. Note that multiplying from the left, we have to transpose the matrices. So in our case, we use "matrices of matrices". Then all computations, all choices of bases etc. can be done exactly as in the case with one-pointed algebras. The notation is somewhat more cumbersome because of the matrix expressions, but that is not a problem as one will see from the examples.

\section{Incidence-free examples }
\subsection{Obstruction theory in the Hochscild complex}

Let $\mathcal M=\left(\begin{matrix}k[t_{11}]&k\cdot t_{12}\\k\cdot t_{21}&k[t_{22}]\end{matrix}\right)/(t_{11}t_{12}-t_{12}t_{22}).$ The points are the simple modules, i.e. the points on the two lines, and we are going to compute their local formal moduli using the Hochshild cohomology. As always, we start with the most general points: $$V_1=k[t_{11}]/(t_{11}),\text{}V_2=k[t_{22}]/(t_{22}),$$ $V=V_1\oplus V_2$, $\mathcal V=\{V_1,V_2\}$. We know the tangent space, $$\HH^1(\mathcal M,\Hom_k(V_i,V_j))=\left(\begin{matrix}k\cdot d_{11}&k\cdot d_{12}\\k\cdot d_{21}&k\cdot d_{22}\end{matrix}\right).$$ It remains to compute the cup and Massey products. The following computation is just to get hold on these in the Hochschild cohomology: To do the illustration we take the most easy case $S=S_2=k[t]/(t^3)\rightarrow k[\varepsilon]=S_1$. Then a lifting $V_S$ is given by $\sigma_S:A\rightarrow\End_k(V_S)$, inducing, for each $a\in A$, $\sigma_S(a):V\rightarrow S\otimes_k V.$ We write $\sigma_S(a)(v)=1\otimes av+t\otimes\sigma_a(v)$. For the following, it is essential to notice that we work with right $A$-modules. The coboundary condition is then given by the following:

$$\begin{aligned}
\sigma_S(ab)=\sigma_S(a)\sigma_S(b)&\Leftrightarrow\sigma_S(ab)(v)=\sigma_S(b)(\sigma_S(a)(v))\Leftrightarrow\\
1\otimes vab+t\otimes\sigma_{ab}(v)&=\sigma_S(b)(1\otimes va)+t\otimes\sigma_b(va)\\&+(t\otimes 1)(b\sigma_a(v)+t\otimes\sigma_b(\sigma_a(v))\Leftrightarrow\\
t\otimes\sigma_{ab}(v)&=t\otimes\sigma_b(va)+t\otimes b\sigma_a(v)+t^2\otimes\sigma_b(\sigma_a(v))\Leftrightarrow\\
t\otimes(\sigma_{ab}(v)&-\sigma_b(va)-b\sigma_a(v))-t^2\otimes\sigma_b(\sigma_a(v))=0.
\end{aligned}$$

So, not surprisingly, we see that the "linear term morphism" $\sigma:A\rightarrow\End_k(V,V)$ satisfies
$$a\sigma(b)-\sigma(ab)+\sigma(a)b=0.$$

This says that if $V_S$ is a lifting, then the linear term $\sigma$ is a derivation. Also the second order term is $\sigma(a)\sigma(b)$, which is the cup-product. It is quite clear how to generalize this to the test algebra, thus it makes it possible to define Cup and Generalized Massey Products.

We recall, as above, that for $\phi:S\rightarrow\End_k(V_S),$ $d\phi:S\otimes_k S\rightarrow\End_k(V_S)$ is given by
$$d\phi(s_1\otimes s_2)=s_1\phi(s_2)-\phi(s_1s_2)+\phi(s_1)s_2.$$

Thus, in particular $d(t_{ij}^\ast)(t_{ij}\otimes 1)=d(t_{ij}^\ast)(1\otimes t_{ij})=0$, and so $d(t_{ij}^\ast)=0$. On the other hand, we find, in the case of a free free $k^r$-algebra,

$$d(t_{ij}t_{jk})^\ast(t_{ij}\otimes t_{jk})=-(t_{ij}\otimes t_{jk})^\ast,$$ in fact forcing all cup-products to be zero.

In this situation with a relation, the above is true in all cases but the one case due to the relation with:

$$d(t_{11}t_{12})^\ast=(t_{11}\otimes t_{12})^\ast-(t_{12}\otimes t_{22})^\ast.$$ We find that due to the above, the cup-products are given by

$$\begin{aligned}
d_{t_{11}}\cup d_{t_{12}}(t_{11}\otimes t_{12})&=1\Rightarrow d_{t_{11}}\cup d_{t_{12}}=(t_{11}\otimes t_{12})^\ast\\
d_{t_{12}}\cup d_{t_{22}}(t_{12}\otimes t_{22})&=1\Rightarrow d_{t_{12}}\cup d_{t_{22}}=(t_{12}\otimes t_{22})^\ast
\end{aligned}$$

This says that $$d_{t_{11}}\cup d_{t_{12}}=-d_{t_{12}}\cup d_{t_{22}}.$$ This gives the relation $r(\uu)=u_{11}u_{12}-u_{12}u_{22}$.

Put $$\mathcal U=\left(\begin{matrix}k[u_{11}]&u_{12}\\u_{21}&k[u_{22}]\end{matrix}\right)/(r(\uu)),$$ define $\sigma_S:\mathcal M\rightarrow \End_k(V_{\mathcal U})$ by

$$\sigma_{\mathcal U}(a)(v)=1\otimes\rho(a)(v)+\sum u_{ij}\otimes d_{t_{ij}}+u_{11}u_{12}\otimes(t_{11}t_{12})^\ast.$$

In particular, notice that this gives the initial $k$-algebra back, highlighting the closure theorem in Laudal's exposition \cite{Laudal03}.

We check that this gives the decided lifting, "all the way to the top", and so the local formal moduli in this case is $\mathcal U$, with the proversal deformation defined by liftings to the top. Note that we have to choose an actual, nonzero, linear morphism killing the obstruction at each step. It must be proved that this is possible, and it follows by induction. In the next example, working with matrices and complexes, this is no longer necessary.

\subsection{Deformation theory in the Yoneda complex: The pair (line, point)}

We consider the plane $A=k[x,y]$, the $x$-axis $V_1=k[x,y]/(y)$, and the origin $V_2=k[x,y]/(x,y)$.

We put $V=V_1\oplus V_2$ and constructs the following resolution

$$\xymatrix{0&\ar[l]V&\ar[l]_{\begin{tiny}\left(\begin{matrix}1&0\\0&1\end{matrix}\right)\end{tiny}}A\oplus A&\ar[l]_{\begin{tiny}\left(\begin{matrix}{\left(\begin{matrix}y\end{matrix}\right)}&0\\0&{\left(\begin{matrix}x&y\end{matrix}\right)}\end{matrix}\right)\end{tiny}}A\oplus A^2&\ar[l]_{\begin{tiny}\left(\begin{matrix}{\left(\begin{matrix}0\end{matrix}\right)}&0\\0&{\left(\begin{matrix}y\\-x\end{matrix}\right)}\end{matrix}\right)\end{tiny}}A&\ar[l]0}$$

\begin{lemma}\label{cllemma}
The tangent space of the deformation functor of  $V$ is the following:
$$\begin{aligned}\Ext^1_A(V_1,V_1)&=V_1,\text{}\Ext^1_A(V_1,V_2)=V_2\cong k,\\
\Ext^1_A(V_2,V_1)&\cong k,\text{}\Ext^1_A(V_2,V_2)=k^2.
\end{aligned}$$
\end{lemma}

\begin{proof}
Taking the $\Hom(-,V)$ of the sequence, computing componentwise, we get:

$\Ext^1_A(V_1,V_1):$ $\xymatrix{V_1\ar[r]^0&V_1\ar[r]&0}\Rightarrow\Ext^1_A(V_1,V_1)=V_1,$

$\Ext^1_A(V_1,V_2):$ $\xymatrix{V_2\ar[r]^0&V_2\ar[r]&0}\Rightarrow\Ext^1_A(V_1,V_1)=V_2\cong k,$

$\Ext^1_A(V_2,V_1):$ $\xymatrix{V_1\ar[r]^{\begin{small}\left(\begin{matrix}x\\0\end{matrix}\right)\end{small}}&
V_1^2\ar[r]^{\begin{small}\left(\begin{matrix}0&-x\end{matrix}\right)\end{small}}&V_1\ar[r]&0}.$ We give the straight forward computation:

$$\left(\begin{matrix}0&-x\end{matrix}\right)\left(\begin{matrix}\overline{v}_1\\\overline{v}_2\end{matrix}\right)=0\Longleftrightarrow -x\overline{v}_2=0\Leftrightarrow -xv_2=hy\Leftrightarrow v_2\in(y)\Leftrightarrow\overline{v}_2=0.$$ Thus the kernel is the set of elements on the form $(\overline{v}_1,0)$, the image is the elements on the form $(xv,0)$ so that $\Ext^1_A(V_2,V_1)=\langle(\alpha,0)\rangle\cong k.$

$\Ext^1_A(V_2,V_2):$ $\xymatrix{V_2\ar[r]^0&V_2^2\ar[r]^0&V_2}\Rightarrow\Ext^1_A(V_2,V_2)\cong V_2^2\cong k^2.$
\end{proof}

We have computed the tangent space of the deformation functor. A line can of course be deformed flatly into any other curve passing through the origin. This is the result of $\Ext^1_A(V_1,V_1)=V_1.$ In this example, we are interested in deformations of lines, thus we will only consider the deformations of the line that are also lines. This equals the linear homogeneous deformations, and we choose $x\in V_1=\Ext^1_A(V_1,V_1)_{(1)}$ as our tangent direction. So for this example, our free non-commutative algebra with $H/\rad(H)^2=S/\rad(S)^2$ is $$S=\left(\begin{matrix}k\langle t_{11}\rangle&t_{12}\\t_{21}&k\langle t_{22}(1),t_{22}(2)\rangle\end{matrix}\right).$$

Now we will give the Yoneda representation of the tangent space. Recall that the Yoneda complex for a resolution $L.$ of $M$ is given as
$$\Hom^p_A(L.,L.)=\{\xi_{i}:L_{i+p}\rightarrow L_i\}_i,$$ with differential $d^p:\Hom^p(L.,L.)\rightarrow\Hom^{p+1}(L.,L.)$ given by
$$d^p(\{\xi\}=\{\xi_i\circ d-(-1)^p d\circ\xi_{i+1}\}.$$

We  illustrate this with the following obvious diagrams:

$$\begin{small}\xymatrix{0&\ar[l]V_1&\ar[l]A&A\ar[l]_{\cdot y}\ar@{-->}[dll]\ar[dl]_{\cdot x}&\ar[l]0\\0&\ar[l]V_1&\ar[l]A&\ar[l]A&\ar[l]0}\end{small}
\text{ }\begin{small}\xymatrix{0&\ar[l]V_1&\ar[l]A&A\ar[l]_{\cdot y}\ar@{-->}[dll]\ar[dl]_{\cdot 1}&\ar[l]0\\0&\ar[l]V_2&\ar[l]A&\ar[l]A^2&\ar[l]A&\ar[l]0}\end{small}$$

$$\xymatrix{0&\ar[l]V_2&\ar[l]A&\ar@{-->}[dll]\ar[dl]_{\begin{tiny}\left(\begin{matrix}1&0\end{matrix}\right)\end{tiny}}
\ar[l]_{\left(\begin{matrix}x&y\end{matrix}\right)}A^2
&\ar[dl]^{-1}\ar[l]_{\left(\begin{matrix}y\\-x\end{matrix}\right)}A&\ar[l]0\\0&\ar[l]V_1&\ar[l]A&\ar[l]_{y}A&\ar[l]0}
$$

$$\xymatrix{0&\ar[l]V_2&\ar[l]A&\ar@{-->}[dll]\ar@<1ex>[dl]_{\begin{tiny}\left(\begin{matrix}1&0\end{matrix}\right)\end{tiny}}
\ar[dl]^{\begin{tiny}\left(\begin{matrix}0&1\end{matrix}\right)\end{tiny}}\ar[l]_{\left(\begin{matrix}x&y\end{matrix}\right)}A^2
&\ar[l]_{\left(\begin{matrix}y\\-x\end{matrix}\right)}\ar@<1ex>[dl]_{\begin{tiny}\left(\begin{matrix}0\\-1\end{matrix}\right)\end{tiny}}\ar[dl]^{\begin{tiny}\left(\begin{matrix}1\\0\end{matrix}\right)\end{tiny}}A&\ar[l]0\\0&\ar[l]V_2&\ar[l]A&\ar[l]^{\left(\begin{matrix}x&y\end{matrix}\right)}A^2
&\ar[l]^{\left(\begin{matrix}y\\-x\end{matrix}\right)}A&\ar[l]0.}$$

Given the Yoneda representation of the chosen tangent space, we might compute the 2. order Massey Products (the cup products). To make clear how morphisms are composed, notice the following illustrative way of thinking:

$$L_2\rightarrow S\otimes_kL_1\rightarrow S\otimes_k S\otimes_k L_0$$ sends $l$ to the following sequence: $$l\mapsto\ut_1\otimes\alpha_{\ut_1}(l)\mapsto\ut_1(\ut_2\otimes\alpha_{\ut_2}(\alpha_{\ut_1}(l)))=\ut_1\ut_2\otimes\alpha_{t_2}(\alpha_{\ut_1}(l)).$$
(So multiplying column vectors from the left gives us again that the last should be the first). We get the following results:

$$\begin{aligned}\langle t_{11}^2\rangle&=0,\text{}\langle t_{11}t_{12}\rangle=0,\text{}\langle t_{22}(1)t_{21}\rangle=0,\\
\langle t_{22}(1)^2\rangle&=0,\text{}\langle t_{22}(2)t_{22}(1)\rangle=1,\text{}\langle t_{12}t_{21}\rangle=0,\\
\langle t_{12}t_{22}(1)\rangle&=0,\text{}\langle t_{22}(2)t_{21}\rangle=1,\text{}\langle t_{22}(2)^2\rangle=0,\\
\langle t_{21}t_{12}\rangle&=-1,\text{}\langle t_{11}^2\rangle=0,\text{}\langle t_{12}t_{22}(2)\rangle=0,\\
\langle t_{21}t_{11}\rangle&=-x,\text{}\langle t_{22}(1)t_{22}(2)\rangle=-1,\text{}\langle t_{21}t_{11}(2)\rangle=0.
\end{aligned}$$

Then this is nearly as simple as it can be, every cup product are zero or a base element, as far as $\langle t_{21}t_{11}\rangle=-x=0\in\Ext^2_A(V_2,V_1),$ forcing us to choose the following 2. order definining system:

$$\xymatrix{0&\ar[l]V_2&\ar[l]A&\ar[l]_{\left(\begin{matrix}x&y\end{matrix}\right)}\ar[dl]_{\begin{small}\left(\begin{matrix}0&-1\end{matrix}\right)\end{small}}A^2&\ar[l]_{\left(\begin{matrix}y\\-x\end{matrix}\right)}\ar[dl]^0\ar@{-->}[dll]_{-x}A&\ar[l]0\\
0&\ar[l]V_1&\ar[l]A&\ar[l]^{y}A&\ar[l]0&\ar[l]0.}$$ This means that $\alpha_{t_{21}t_{11}=}\begin{cases}\left(\begin{matrix}0&-1\end{matrix}\right)\\0\end{cases},$ and that all the rest of the 2. order defining systems can be chosen to be $0$. The only 3. order Massey products to be computed are then:
$$\langle t_{21}t_{11}t_{11}\rangle=0,\text{}\langle t_{22}(1)t_{21}t_{11}\rangle=1,\text{}\langle t_{22}(2)t_{21}t_{11}\rangle=0.$$

As the remaining differences are trivial cohomology classes, identically $0$, we have the following result

\begin{proposition}The versal base space of a line through the origin and a point is
$$H\simeq\left(\begin{matrix}k[t_{11}]&t_{12}\\t_{21}&k\langle t_{22}(1),t_{22}(2)\rangle\end{matrix}\right)/(t_{22}(2)t_{22}(1)-t_{22}(1)t_{22}(2)-t_{21}t_{12},t_{22}(2)t_{21}+t_{22}(1)t_{21}t_{11})$$
\end{proposition}

We can interpret this result geometrically: Putting $t_{12}=t_{21}=0$ we get families

(1,1): $k[x,y]/(t_{11}x+y),\text{}t_{11}\in k$: Lines with slope $t_{11}$.

(2,2): $k[x,y]/(x+t_{22}(1),y+t_{22}(2)$: Points $(-t_{22}(1),-t_{22}(2))$.

\noindent The variables $t_{12},t_{21}$ tells how the objects are related at tangent level.

$$t_{22}(2)t_{22}(1)-t_{22}(1)t_{22}(2)-t_{21}t_{12}$$ gives no forced tangent relations, it is just a description of the geometry.

$$t_{22}(2)t_{21}+t_{22}(1)t_{21}t_{11}\Rightarrow t_{21}(t_{11}t_{22}(1)+t_{22}(2))=t_{21}(t_{11}x+y)$$ means that the constant $\ext$-locus consists of points on the line, (the line). So this is the moduli of all pairs of a point and a line through the origin. In particular, the $\Ext$-dimension is correct for points on the line.

Maybe, to the above, we should give the computation:
$$\begin{aligned}\delta(t_{22}(2)t_{21}+t_{22}(1)t_{21}t_{11})=0\Leftrightarrow t_{22}(2)\delta(t_{21})+t_{22}(1)\delta(t_{21}t_{11})=0\Leftrightarrow\\
t_{22}(2)\delta(t_{21})+t_{22}(1)\delta(t_{21})t_{11}=0\Leftrightarrow\delta(t_{21})(t_{22}(2)+t_{22}(1)t_{11})
\end{aligned}$$

\section{Deformations and interactions}
\subsection{The category $a_\Gamma$}

\begin{definition} A \textit{diagram} $\uc$ of right $A$-modules consists of a family $|\uc|$ of right $A$-modules, together with a set $\Gamma(V,W)\subseteq\Hom_A(V,W)$ of $A$-module homomorphisms for each pair of modules $ V,W\in |\uc|$.
\end{definition}

If $A=k$, this is called a representation of the corresponding quiver, i.e. the quiver with $|\uc|$ as set of nodes and $\Gamma=\underset{V,W\in |\uc|}\bigcup\Gamma(V,W)$ as arrows. Throughout, we will just use the notation $\Gamma$ for the corresponding quiver.

We let $k[\Gamma]$ denote the the quiver algebra. By definition, the quiver algebra is the $k$-algebra generated by all finite paths $\gamma_1\gamma_2\cdots\gamma_i\gamma_{i+1}\cdots\gamma_n$ such that the head of $\gamma_i$ is the tail of $\gamma_{i+1}$. Note that $e_i$, the identity at the node $i$, is considered as a finite path. Thus $k[\Gamma]$ is isomorphic to $k^r[\Gamma]$.

\begin{example}
$$\xymatrix{V_1\ar[r]^{\gamma_{12}}&V_2\ar[d]^{\gamma_{23}}\\&V_3.
}$$ Then $k[\Gamma]$ is the matrix algebra consisting of all matrices on the form
$\left(\begin{matrix}k&k&k\\ 0&k&k\\0&0&k\end{matrix}\right).$
\end{example}

\begin{example}
$$\xymatrix{V_1\ar[r]^{\gamma_{12}}\ar[dr]_{\gamma_{13}}&V_2\ar[d]^{\gamma_{23}}\\&V_3.
}$$ Then $k[\Gamma]$ is the matrix algebra consisting of all matrices on the form
$\left(\begin{matrix}k&k&k\oplus k\\ 0&k&k\\0&0&k\end{matrix}\right).$
\end{example}

\begin{definition} The category $a_\Gamma$ of pointed Artinian $\Gamma$-algebras is the category of $k[\Gamma]$-algebras fitting into the diagram $$\xymatrix{k[\Gamma]\ar[r]\ar[dr]_{\Id}&S\ar[d]^\rho\\&k[\Gamma],}$$ and such that $\rad(S)^n=(\ker\rho)^n=0$. The morphisms of $a_\Gamma$ are the commuting $k[\Gamma]$-homomorphisms.
\end{definition}

We notice that when $\Gamma=\emptyset$, that is a diagram with a trivial quiver (no morphisms but the identities at each node), then $k[\Gamma]=k^r$ so that this definition is a generalization of the $r$-pointed algebras. Also notice that by  Lemma \ref{Fausklemma}, $S$ has exactly $r$ simple modules.

For the deformation theory, we recall that in the discrete situation, i.e. $\Gamma=\emptyset$, we considered $V=\underset{i=1}{\overset r\oplus}V_i$ as an $A^r$-module, and that a lifting $V_S$ of $V$ to $S\in\ob(a_r)$ is a $S\otimes_{k}A$-module satisfying $k^r\otimes_SV_S\cong V$. So what we do, is to consider $V$ as an $k^r\otimes_k A$-module. This makes perfectly sense also in the situation with nontrivial quivers:

We consider $\uc$ as a $k[\Gamma] \otimes_k A$-module, for short, as an $A[\Gamma]$-module, by letting the elements in $\Gamma$ act by right multiplication. That is, an element $v=(v_1,\dots,v_r)\in V$ is given the action $$v\cdot\gamma_{ij}(l)=\left(\begin{matrix}v_1&\cdots&v_r\end{matrix}\right)\left(\begin{matrix}&\vdots&\\
\ldots&\gamma_{ij}(l)&\ldots\\&\vdots&\end{matrix}\right)=\left(\begin{matrix}\cdots&v_i\cdot\gamma_{ij}(l)&\cdots\end{matrix}\right)$$ where we let $v_i\cdot\gamma_{ij}(l)=\gamma_{ij}(l)(v_i)$. So $V$ is a right $A$, right $k[\Gamma]$-module, and because $\Gamma$ consists of $A$-linear morphisms, these actions commute. Thus $V$ is an $A[\Gamma]$-module.

We want to generalize the deformation functor $\Def_V:a_r\rightarrow\Sets$ to the category $a_\Gamma$. A deformation of the diagram $\uc$ to an object $S$ in $a_\Gamma$ should be a deformation $V_S$ which is a deformation of $V=|\uc|$ to $S$ as an object in $a_r$, but it should also lift the morphisms in the diagram, i.e. the quiver $\Gamma$ of $\uc$. Here, $V_S$ a lifting of $V=|\uc|$ to $S$ as object in $a_r$, means the natural restriction to $S_r=S/\Gamma$.

\begin{definition} Let $\uc$ be a diagram of $A$-modules. We define $\Def_{\uc}:a_\Gamma\rightarrow\Sets$ by letting a deformation, or lifting, of $\uc$ to $S$ be an $S\otimes_k A$-module $V_S$, flat over $S$, such that $k[\Gamma]\otimes_S V_S\cong\uc$, as an $A[\Gamma]$-module.
\end{definition}

Notice that in the discrete situation, deformation flatness of the lifting $V_S$ of $V$ to $S$ is equivalent to $V_S\cong (S_{ij}\otimes_k V_j)$ as $S$-module, or in fact $V_S\cong_S S\otimes_{k^r}V.$ In the situation with incidences,  for an $A[\Gamma]$-module $V$ we get deformation flatness defined as $V_{S}\cong_{S} S\otimes_{k[\Gamma]}V$.

\begin{example} $A=k$, $V_1=k$, $V_2=k^2$, $\gamma_{12}=\left(\begin{matrix}1\\1\end{matrix}\right)$.
\end{example}

\begin{example} $A=k[x]$, $V_1=A$, $V_2=A/(x-\alpha)$, $\gamma_{12}=\kappa$.
\end{example}

\begin{example} Consider the diagram $\uc=\xymatrix{A\ar[r]^{(1,1)}&A^2}$. Then $k[\Gamma]$ is the $k$-algebra generated over $k$ by the matrices $e_{11}=\left(\begin{matrix}1&0\\0&0\end{matrix}\right)$, $e_{22}=\left(\begin{matrix}0&0\\0&1\end{matrix}\right)$ and $e_{11}\gamma_{12}=\left(\begin{matrix}0&\gamma_{12}\\0&0\end{matrix}\right)$ which we write $k[\Gamma]=\left(\begin{matrix}k&\gamma_{12}\\0&k\end{matrix}\right)$ for short. Then $A[\Gamma]=\left(\begin{matrix}A&\gamma_{12}\\0&A\end{matrix}\right)$.
\end{example}

\begin{lemma} There is an equivalence between the category of finite diagrams $\uc$ ($|\uc|=\{V_1,\dots,V_r\}$) and quiver algebras $A[\Gamma]$ with a finite set of nodes.
\end{lemma}

\begin{proof}
Let $\uc=(|\uc|=\{V_i\}_1^r,\Gamma=\{\gamma_{ij}\}_{1\leq i,j\leq r})$ be a diagram. Then $V=\oplus_{i=1}^r V_i$ with the natural right action ($V$ considered as a diagonal matrix) is an $A[\Gamma]$-module. Conversely, an $A[\Gamma]$-module $V$ is also an $A^r$-module. Put $\frak a_i=\{(a_1,\dots,\hat a_i,\dots a_r)\}$, $1\leq i\leq r$. Then $V_i=V/\frak a_i V$ is an $A$-module, and we let the quiver morphism $\gamma_{ij}:V_i\rightarrow V_j$ be defined by $\gamma_{ij}(v_i)=v\cdot e_i\gamma_{ij}$, where $v$ is any element in the inverse image of $\kappa:V\rightarrow V_i$.
\end{proof}

\subsection{Deformation theory in $a_\Gamma$}

In this subsection we fix a diagram $\uc=(|\uc|=\{V_1,\dots,V_r\},\Gamma)$ and let $V$ be the corresponding $A[\Gamma]$-module. For a $k[\Gamma]$-algebra $S$ in $a_\Gamma$, a deformation of $V$ to $S$ is an $S\otimes_k A$-module $V_S$, flat over $S$, such that $k[\Gamma]\otimes_S V_S\cong V$ as $A[\Gamma]=k[\Gamma]\otimes_k A$-module. Two deformations are equivalent, $V_S\sim V_S^\prime$, if they are isomorphic over $S$, e.g. there exists an isomorphism $\iota:V_S\rightarrow V_S^\prime$ commuting with the induced isomorphism $k[\Gamma]\otimes_S V_S\cong k[\Gamma]\otimes_S V_S^\prime$, that is, the diagram $$\xymatrix{V_S\ar[dr]\ar[rr]^\iota&&V_S^\prime\ar[dl]\\&V&}$$ commutes.

\begin{lemma}
$V_S\in\Def^\Gamma_V(S)$ is $S$-flat if and only if $V_S\cong S\otimes_{k[\Gamma]} V$ as $S$-module.
\end{lemma}

\begin{proof}
This follows exactly as in the discrete situation; for $R\in a_r$, $V_R\in\Def_V(R)$ is $R$-flat if and only if $V_R\cong R\otimes_{k^r} V$ as $R$-module.
\end{proof}

Thus, a deformation, or lifting of $V$ to $S$ is an $A$-module structure on $S\otimes_{k[\Gamma]}V$, commuting with the action of $S$ (and then the induced action of $k[\Gamma]$). Following step by step the discrete situation, this is, for every $a\in A$ to give an action morphism $\sigma_a:V\rightarrow S\otimes_{k[\Gamma]}V$, commuting with the $k[\Gamma]$-action. There is a $k^r$-morphism $\kappa:S\otimes_{k[\Gamma]}V\twoheadrightarrow S\otimes_{k^r}V$ given by $\kappa(\gamma_{ij}\otimes v)=1\otimes\gamma_{ij}v$, e.g., the $\Gamma$-action on $S\otimes_{k^r}V$ is right $\Gamma$ action on $V$. So this is equivalent to, for each $a\in A$, to give an action morphism $\sigma_a:V\rightarrow S\otimes_{k^r}V$ commuting with all $\gamma\in\Gamma$. The obstruction theory is then exactly as before, except that the cohomology controlling the deformations is $\Ext_A^\Gamma(V_i,V_j)$, the left derived functor of $\Hom_A^\Gamma(V,-)$.
Notice in particular that the test-algebra in the incidence situation is $$k^\Gamma[\varepsilon]=k[\Gamma]\otimes_{k^r}(t_{ij})/(t_{ij})^2.$$

As in the discrete situation, the obstruction calculus can be performed in the Hochschild cohomology; the $\sigma$'s give a homomorphism $$\sigma:A\rightarrow\Hom^\Gamma_A(A,\Hom_k(V_i,V_j))\otimes_kI\subseteq\Hom_A(A,\Hom_k(V_i,V_j))\otimes_kI$$ where the superscript $\Gamma$ denotes the subspace of morphisms commuting with all $\gamma\in\Gamma$, and $I$ is the kernel in a particular small morphism $\pi:S\twoheadrightarrow R.$ All computation are identical, we should only be sure they respects the action of $\Gamma$.

Experience proves that, in some situations, it is easier to work with free resolutions of modules. The computations in the Hochschild cohomology then translates as follows:

Choose resolutions $0\leftarrow V_i\rightarrow L_{.}^i.$ We can lift $\Gamma$ to the components in the respective projective resolutions, so that $L_i=\underset{j=1}{\overset r\oplus} L_i^j$ becomes an $A[\Gamma]$-module as well as $V=\underset{j=1}{\overset r\oplus} V_j$:

$$\xymatrix{0&V_i\ar[l]\ar[d]^{\gamma_{ij}}&L_0^i\ar[l]\ar[d]^{\gamma_{ij}}&
L_1^i\ar[l]_{d_0^i}\ar[d]^{\gamma_{ij}}&L_2^i\ar[l]_{d_1^i}\ar[d]^{\gamma_{ij}}&\cdots\ar[l]_{d_2^i}
\\0&V_j\ar[l]&L_0^j\ar[l]&L_1^j\ar[l]_{d_0^j}&L_2^j\ar[l]_{d_1^j}&\cdots\ar[l]_{d_2^j}}$$

In the discrete situation, we worked in the Yoneda complex $\Hom^._A(L.,L.)$, using the quasi isomorphism $$\iota:\Hom^._A(L.,L.)\rightarrow\Hom_A(L.,V).$$ Lifting the action of $\Gamma$ as above, we get the natural action of $\Gamma$ on  $\Hom^._A(L.,L.)$ and $\Hom_A(L.,V)$, giving us the possibility to consider $\Hom_A^\Gamma(L.,V)$ and $\Hom^{.,\Gamma}_A(L.,L.)$, the morphisms that are invariant under $\Gamma$. There is no reason why these two are quasi-isomorphic in general, so we have to take invariant cycles, construct obstructions in $H^2(\Hom^._A(L.,L.))\cong H^2(\Hom(L.,V))\cong\HH^2(A,\End_k(V))$, knowing that the resulting class is an invariant, that is an element in $${}^\Gamma\HH^2(A,\End_k(V))\subseteq\Ext^2_A(L.,V).$$

So we start choosing bases $\{(t_{ij}^{\ast}(l))\}\subset{}^\Gamma\Ext^1_A(V,V)$, we let the test algebra be $$S=k[\Gamma]\otimes_{k^r}(t_{ij}^\ast(l)),$$
and we do the computations exactly as before.

\section{Examples with incidences}

\subsection{A fixed line and a point}

We consider the following diagram of $A$-modules: $$A=k[x]=V_1\overset{\gamma_{12}}\rightarrow k[x]/(x)=k.$$ We then have the following resolutions
$$\xymatrix{0&V_1\ar[l]\ar[d]_{\gamma_{12}}&A\ar[l]\ar[d]&0\ar[l]\ar[d]&&\\0&V_2\ar[l]&A\ar[l]&A\ar[l]^{\cdot x=d_1}&0\ar[l]}$$ which immediately gives the following tangent spaces:
$\Ext^1_A(V_1,V_1)=\Ext^1_A(V_1,V_2)=0$, $\Ext^1_A(V_2,V_1)=\Ext^1_A(V_2,V_2)=k$. Thus we have $S_2=\left(\begin{matrix}k&0\\t_{21}&k[t_{22}]\end{matrix}\right)$.

To make things clear, the liftings to the second level is given by

$$\xymatrix{
0&\mathcal V_2\ar[l]\ar[d]&S_2\otimes_{k^2} L_0\ar[l]\ar[d]&S_2\otimes_{k^2} L_1\ar[l]_{\mathcal M_2}\ar[d]&0\ar[l]\\
0&\ar[l]{\left(\begin{matrix}V_1&0\\0&V_2\end{matrix}\right)}&\ar[l]
{\left(\begin{matrix}L_0^1&0\\0&L_0^2\end{matrix}\right)}&\ar[l]
{\left(\begin{matrix}L_1^1&0\\0&L_1^2\end{matrix}\right)}&\ar[l] 0}$$

with the universal family restricted to the tangent space given by $$\mathcal M_2=\left(\begin{matrix}1\otimes d_{11}^1&0\\t_{21}\otimes\xi_{21}&1\otimes d_{22}^1+t_{22}\otimes\xi_{22}\end{matrix}\right).$$

As we see, As $\Ext^2_A(V_i,V_j)=0$ all higher order Massey products are zero, and we are through immediately.

\subsection{Example: A line through the origin and a point on the line}

Inspired by the moduli (stack) of curves with fixed points, we consider the first step of the moduli problem of parametrizing homogeneous curves with $r$ fixed points:

Parameterize pairs $(L,p)$ where $L$ is a line through the origin in the plane and $p$ is a point on the line $L$.

We consider the plane $k[x,y]$, the $x$-axis $V_1=k[x,y]/(y)$, and the origin $V_2=k[x,y]/(x,y)$.

Inside the moduli of the pairs $(L,p)$, a line and a point, lies the moduli space of pairs $(L,p)$ with $p$ a point on the line $L$. This is to say algebraically that there is a homomorphism $\gamma:A(L)\rightarrow k(p)$, where $A(L)$ denotes the affine ring of $L$. For this subspace of moduli, we use the corresponding notations:

$$V_1=A/(y)\overset{\gamma_{12}}\rightarrow A/(x,y)=V_2,\text{}A=k[x,y].$$

We lift the quotient morphism, which is our incidence in this example, to the resolution $0\leftarrow V=V_1\oplus V_2\leftarrow L.=L^1.\oplus L^2.$ of $V$ according to the following

$$\xymatrix{0&\ar[d]_{\gamma_{12}}V_1\ar[l]&\ar[d]_{\Id}A\ar[l]&\ar[d]_{\left(\begin{matrix}0\\1\end{matrix}\right)}A\ar[l]_y&0\ar[l]\\0&V_2\ar[l]&A\ar[l]&A^2
\ar[l]^{\left(\begin{matrix}x&y\end{matrix}\right)}&A\ar[l]&0\ar[l]}$$

Then we start computing, taking the incidence into consideration. Let $\phi=\left(\begin{matrix}\phi_{11}&\phi_{21}\\
\phi_{12}&\phi_{22}\end{matrix}\right)\in\Ker(\Hom_A(L_1,V)\rightarrow\Hom_A(L_2,V))$. From the computations in the previous subsection, we then know

$$\begin{aligned}
\phi_{11}&=v\in V_1\\
\phi_{21}&=(v,0)\in V_1^2\\
\phi_{12}&=\alpha\in V_2\cong k\\
\phi_{22}&=(\alpha,\beta)\in V_2^2\cong k^2.
\end{aligned}$$

For $\phi$ to be invariant under the action of $\Gamma$, i.e. $\phi\in\Hom^\Gamma_A(L_1,V)$, the diagram
$$\xymatrix{L_1\ar[d]_{\phi}&L_1\ar[d]^{\phi}\ar[l]_{\gamma_{12}}\\V&V\ar[l]^{\gamma_{12}}}$$
must be commutative. We get
$\phi\circ\gamma_{12}=\gamma_{12}\circ\phi\Leftrightarrow$

$$\left(\begin{matrix}\phi_{11}&\phi_{21}\\
\phi_{12}&\phi_{22}\end{matrix}\right)\left(\begin{matrix}0&0\\
\gamma_{12}&0\end{matrix}\right)=\left(\begin{matrix}0&0\\\gamma_{12}&0\end{matrix}\right)\left(\begin{matrix}\phi_{11}&\phi_{21}\\
\phi_{12}&\phi_{22}\end{matrix}\right)\Leftrightarrow\left(\begin{matrix}\phi_{21}\circ\gamma_{12}&0\\
\phi_{22}\circ\gamma_{12}&0\end{matrix}\right)=\left(\begin{matrix}0&0\\
\gamma_{12}\circ\phi_{11}&\gamma_{12}\circ\phi_{12}\end{matrix}\right)$$
which gives the equations
$$
\begin{aligned} \phi_{21}\circ\gamma_{12}&=0\Leftrightarrow\left(\begin{matrix}v&0\end{matrix}\right)\left(\begin{matrix}0\\1\end{matrix}\right)=0,\\
\gamma_{12}\circ\phi_{21}&=0\Leftrightarrow\phi_{21}=(v,0),\text{}v\in(x)\subset V_1,\\
\phi_{22}\circ\gamma_{12}&=\gamma_{12}\circ\phi_{11}\Leftrightarrow
\left(\begin{matrix}\alpha&\beta\end{matrix}\right)\left(\begin{matrix}0\\1\end{matrix}\right)=
0\Leftrightarrow \phi_{22}=(\alpha,0),\text{}\alpha\in k.\end{aligned}$$

We want to divide out by $\Ima(\Hom^\Gamma_A(L_0,V)\rightarrow\Hom^\Gamma_A(L_1,V)).$ For this, there is only one point of interest;

for $\psi=\left(\begin{matrix}\psi_{11}&\psi_{21}\\\psi_{12}&\psi_{22}\end{matrix}\right)$, we have that $d\psi=\left(\begin{matrix}\psi_{11}&\psi_{21}\\\psi_{12}&\psi_{22}\end{matrix}\right)\left(\begin{matrix}d_1&0\\0&d_2\end{matrix}\right)
=\left(\begin{matrix}0&(x,0)\psi_{21}\\0&0\end{matrix}\right)$.

An element $\psi=\left(\begin{matrix}\psi_{12}&\psi_{21}\\\psi_{12}&\psi_{22}\end{matrix}\right)$ is invariant if and only if $\psi\circ\gamma_{12}=\gamma_{12}\circ\psi\Leftrightarrow$

$$\begin{aligned}\left(\begin{matrix}\psi_{11}&\psi_{21}\\\psi_{12}&\psi_{22}\end{matrix}\right)
\left(\begin{matrix}0&0\\\gamma_{12}&0\end{matrix}\right)&=
\left(\begin{matrix}0&0\\\gamma_{12}&0\end{matrix}\right)
\left(\begin{matrix}\psi_{11}&\psi_{21}\\\psi_{12}&\psi_{22}\end{matrix}\right)\Leftrightarrow\\
\left(\begin{matrix}\psi_{21}&0\\\psi_{22}&0\end{matrix}\right)&=\left(\begin{matrix}0&0\\\gamma_{12}\circ\psi_{11}&\gamma_{12}\circ\psi_{21}\end{matrix}\right)
\Leftrightarrow\psi_{21}=0,\end{aligned}$$ Implying that the invariant image is the zero space.

This means that there is additional deformations in the case with incidences, and $\Ext^1_A(V_1,V_1)^\Gamma$ is infinite dimensional. Now, we choose a basis for the tangent space, contained in the case with incidences (notice that $\phi_{21}$ is killed by $\Hom_A(L_0,V)$ forgetting the incidences), i.e., we choose the following, invariant tangent space:

$$T_A^\Gamma=\{\phi|\phi=\left(\begin{matrix}\alpha_1 x&0\\\alpha_2&{\left(\begin{matrix}\alpha_3&0\end{matrix}\right)}
\end{matrix}\right),\text{}\alpha_i\in k, 1\leq i\leq3\}.$$

The Yoneda representations are given by the following diagrams:

$$\xymatrix{0&V_1\ar[l]&A\ar[l]&A\ar[l]\ar[dll]_x\ar[dl]_x&0\ar[l]\\0&V_1\ar[l]&A\ar[l]&A\ar[l]&0\ar[l]}$$

$$\xymatrix{0&V_1\ar[l]&A\ar[l]&A\ar[l]\ar[dll]_1\ar[dl]_1&0\ar[l]\\0&V_2\ar[l]&A\ar[l]&A^2\ar[l]&A\ar[l]&0\ar[l]}$$

$$\xymatrix{0&V_2\ar[l]&A\ar[l]&A^2\ar[dl]_{\begin{tiny}\left(\begin{matrix}1&0\end{matrix}\right)\end{tiny}}\ar[l]&
A\ar[l]_{\left(\begin{matrix}y\\-x\end{matrix}\right)}\ar[dl]^{\begin{tiny}\left(\begin{matrix}0\\-1\end{matrix}\right)\end{tiny}}&0\ar[l]\\
0&V_2\ar[l]&A\ar[l]&A^2\ar[l]^{\left(\begin{matrix}x&y\end{matrix}\right)}&
A\ar[l]&0\ar[l]}$$

So the 2. order Massey products, the cup products, are

$$\begin{aligned}
\langle t_{11}^2\rangle&=0,\text{}\langle t_{12}t_{22}\rangle=\left(\begin{matrix}0\\-1\end{matrix}\right):L_2^1\rightarrow L_1^2=A\in\Hom_A(L_2^1,V_1),\text{}\\
\langle t_{22}^2\rangle=0.
\end{aligned}$$

Because the only only nonzero element is also necessarily nonzero in cohomology, this means that we end up with the following:

\begin{proposition} The moduli space of the pair $(L,p)$ with $p$ a point on the line $L$, inside the discrete moduli, is
$$\left(\begin{matrix}k[t_{11}]&t_{12}\\0&k[t_{22}(1)]\end{matrix}\right)/(t_{12}t_{22}).$$
\end{proposition}

The geometric interpretation of this is the set of lines with slope $t_{11}$, and the point $(t_{22},t_{11}t_{22})$. The relation just tells that the point has to move along the line.


\begin{thebibliography}{}


\bibitem{Eriksen03}
E. Eriksen.
An Introduction to non-commutative Deformations of Modules. Lect. Notes Pure Appl. Math., {\bf 243} (2005), 90 - 126.



\bibitem{Hartshorne77}
R. Hartshorne.
Algebraic geometry
Graduate Texts in Mathematics, No. 52
New York
1977
ISBN 0-387-90244-9



\bibitem{Laudal03}
O. A. Laudal.
non-commutative Algebraic Geometry.
Rev. Mat. Iberoamericana, {\bf 19 (2)} (2003), 509 - 580.


\bibitem{Schlessinger68}
M. Schlessinger.
Functors of Artin rings.
Trans. Amer. Math. Soc., {\bf 130} (1968), 208 - 222.



\bibitem{Siqveland111}
A. Siqveland.
Geometry of noncommutative $k$-algebras.
J. Gen. Lie Theory Appl., {\bf 5} (2011), 1 - 12.








\end{thebibliography}
\end{document}